\documentclass[psamsfonts]{amsart}

\usepackage{amssymb,amsfonts,amscd}

\usepackage[colorlinks,linktocpage]{hyperref}
\hypersetup{linkcolor=[rgb]{0,0,0.715}}
\hypersetup{citecolor=[rgb]{0,0.715,0}}

\newtheorem{thm}{Theorem}[section]
\newtheorem{cor}[thm]{Corollary}
\newtheorem{prop}[thm]{Proposition}
\newtheorem{lem}[thm]{Lemma}

\newtheorem*{thm1}{Theorem A}

\theoremstyle{definition}
\newtheorem{defn}[thm]{Definition}

\theoremstyle{remark}
\newtheorem{rem}[thm]{Remark}

\makeatletter
\let\c@equation\c@thm
\makeatother
\numberwithin{equation}{section}

\title[]{Nonnegative Ricci curvature and escape rate gap}

\author[]{Jiayin Pan}
\address{The Fields Institute for Research in Mathematical Sciences, Toronto, Ontario, Canada.}
\thanks{The author was partially supported by AMS Simons travel grant when preparing this work.}
\email{jypan10@gmail.com}

\begin{document}

\begin{abstract}
	Let $M$ be an open $n$-manifold of nonnegative Ricci curvature and let $p\in M$. We show that if $(M,p)$ has escape rate less than some positive constant $\epsilon(n)$, that is, minimal representing geodesic loops of $\pi_1(M,p)$ escape from any bounded balls at a small linear rate with respect to their lengths, then $\pi_1(M,p)$ is virtually abelian. This generalizes the author's previous work \cite{Pan_escape}, where the zero escape rate is considered.
\end{abstract}
	
	
\maketitle
	
We study the structure of fundamental group of open manifolds with nonnegative Ricci curvature. In comparison to sectional curvature, we recall that it follows from soul theorem that the fundamental group of any open manifold with nonnegative sectional curvature is virtually abelian, that is, it contains an abelian subgroup of finite index \cite{CG_soul}. Regarding Ricci curvature, Wei constructed open manifolds with positive Ricci curvature and fundamental groups that are torsion-free nilpotent \cite{Wei}. Later, Wilking proved that any finitely generated virtually nilpotent group can be realized as the fundamental group of some open manifold of positive Ricci curvature \cite{Wilk}. Conversely, any finitely generated subgroup of $\pi_1(M)$ has polynomial growth \cite{Mil}, so by Gromov's work \cite{Gro_poly}, it has a nilpotent subgroup of finite index (also see \cite{KW}).

The author discovered that the virtual abelianness/nilpotency of $\pi_1(M)$ is related to where the representing geodesic loops of $\pi_1(M,p)$ are positioned in $M$ \cite{Pan_escape}. For any element $\gamma\in \pi_1(M,p)$, we can choose a geodesic loop based at $p$ representing $\gamma$ of minimal length, denoted by $c_\gamma$. It is known before that Cheeger-Gromoll splitting theorem \cite{CG_split} implies that if all representing geodesic loops are contained in a bounded set, then $\pi_1(M)$ is virtually abelian. However, it is prevalent for representing geodesic loops to escape from any bounded balls in nonnegative Ricci curvature: if $M$ has positive Ricci curvature and an infinite fundamental group, then this escape phenomenon always occurs \cite{SW}. The escape rate $E(M,p)$ introduced in \cite{Pan_escape} measures how fast the representing geodesic loops of $\pi_1(M,p)$ escape from any bounded balls by comparing the size of $c_\gamma$ to its lengths:
$$E(M,p):=\limsup_{|\gamma|\to\infty} \dfrac{d_H(p,c_\gamma)}{|\gamma|},$$
where $|\gamma|$ is the length of $c_\gamma$ and $d_H$ is the Hausdorff distance. For a doubly warped product $M=[0,\infty)\times_f S^{p-1}\times_h S^1$, $E(M,p)$ is determined by the decaying rate of the warping function $h(r)$ (see \cite[Appendix B]{Pan_escape}). As $h(r)$ decreases, a representing geodesic loop would take advantage the thin end to shorten its length, while this also increases its size. Hence the faster $h(r)$ decays, the larger escape rate $(M,p)$ has. As the main result of \cite{Pan_escape}, we proved that if $E(M,p)=0$, then $\pi_1(M)$ is virtually abelian.

In this paper, we further generalize the above mentioned result by proving a universal escape rate gap.
	
\begin{thm1}
	Given $n$, there is a positive constant $\epsilon(n)$ such that for any open $n$-manifold $(M,p)$ of $\mathrm{Ric}\ge 0$, if $E(M,p)\le \epsilon(n)$, then $\pi_1(M,p)$ is virtually abelian.
\end{thm1}

In other words, if $\pi_1(M)$ contains a free nilpotent non-abelian subgroup, then $E(M,p)>\epsilon(n)$, that is, there is a sequence of elements $\gamma_i\in \pi_1(M,p)$ such that $d_H(x,c_{\gamma_i})>\epsilon(n)|\gamma_i|$. The converse of Theorem A is not true in general (see \cite[Appendix B]{Pan_escape}).

The proof of Theorem A involves the study of geometry of $(\widetilde{M},\tilde{p},\pi_1(M,p))$ at infinity, where $(\widetilde{M},\tilde{p})$ is the Riemannian universal cover of $(M,p)$ and $\pi_1(M,p)$ acts on $\widetilde{M}$ as isometries. For any sequence $r_i\to\infty$, we can pass to a subsequence and obtain the following pointed equivariant Gromov-Hausdorff convergence \cite{FY}:
$$(r_i^{-1}\widetilde{M},\tilde{p},\pi_1(M,p))\overset{GH}\longrightarrow (Y,y,G).$$
The limit $(Y,y,G)$ is called an equivariant asymptotic cone of $(\widetilde{M},\pi_1(M,p))$. To illustrate the approach to Theorem A, we first recall the strategy for the zero escape rate case \cite{Pan_escape}, which roughly goes as follows:

$E(M,p)=0;$\\
$\Rightarrow$ $Gy$ is geodesic in $Y$ for any equivariant asymptotic cone $(Y,y,G)$;\\
$\Rightarrow$ $Gy$ is a metric product $\mathbb{R}^k\times Z$, where $Z$ is compact, for any $(Y,y,G)$;\\
$\Rightarrow$ $Gy$ is a standard Euclidean space for any $(Y,y,G)$;\\
$\Rightarrow$ Any nilpotent subgroup $N$ of $\Gamma$ acts as almost translations on $\widetilde{M}$ at large scale;\\
$\Rightarrow$ $\Gamma$ is virtually abelian.

To study the case $E(M,p)\le\epsilon$, we quantify the approach above. As the first step, we will introduce the concept of $\delta$-geodesic, which measures how close a subset is to being geodesic (Definition \ref{def_delta_geo}), and show that $Gy$ is $\delta_\epsilon$-geodesic, where $\delta_\epsilon\to 0$ as $\epsilon\to 0$ (Proposition \ref{limit_delta_geodesic}). 

Regarding the second and third steps above, we use pointed Gromov-Hausdorff closeness to quantify. One may expect that $Gy$ is $\Phi(\epsilon|n)$-close to $\mathbb{R}^k\times Z$, where $Z$ is compact, or $\mathbb{R}^k$ in the pointed Gromov-Hausdorff sense, where $\Phi(\epsilon|n)$ is an unspecified function depending on $\epsilon$ and $n$ with $\lim_{\epsilon\to 0}\Phi(\epsilon|n)=0$. However, this statement has a clear obstruction in the second step: pointed Gromov-Hausdorff closeness cannot distinguish a non-compact space from a compact one with a large diameter. To overcome this, for any equivariant asymptotic cone $(Y,y,G)$, we shall consider an associated family of spaces $\{(sY,y,G)\}_{s>0}$, where $(sY,y,G)$ means scaling $(Y,y,G)$ by $s$. We shall apply Cheeger-Colding quantitative splitting theorem \cite{CC1} to $(Y,y,G)$ only when almost splitting holds for all $(sY,y,G)$. With this idea, we show that for any $(Y,y,G)$, either $Gy$ in $(sY,y,G)$ is close to $\mathbb{R}^k$ for all $s>0$, or there is some $s>0$ such that $Gy$ in $(sY,y,G)$ is close to a product $\mathbb{R}^k\times Z_s$ with the diameter of $Z_s$ being around $1$ (see Proposition \ref{split_all_scales}). Next, we further rule out the compact factor $Z_s$; more precisely, we prove the following (also see Definition \ref{def_GH_subset} and Proposition \ref{converse}):

\begin{thm}\label{almost_eu_orbit}
	Let $(M,p)$ be an open $n$-manifold with $\mathrm{Ric}\ge 0$ and $E(M,p)\le\epsilon$. Then there is an integer $k$ such that for any $(Y,y,G)\in\Omega(\widetilde{M},\Gamma)$, we have
	$$d_{GH}((Y,y,Gy),(\mathbb{R}^k\times X,(0,x),\mathbb{R}^k\times \{x\}))\le \Phi(\epsilon|n),$$
	where $(X,x)$ is a length space that depends on $(Y,y)$.
\end{thm}	

The proof of Theorem \ref{almost_eu_orbit} relies on a critical rescaling argument, which is effective to prove uniformity among all equivariant asymptotic cones. This type of argument is first introduced by the author in \cite{Pan_eu} and also applied in \cite{Pan_al_stable,Pan_escape}.

We organize the paper as follows. We start with preliminaries in Section \ref{sec_pre}. In Section \ref{sec_geo}, we introduce the notion of {$\delta$-geodesic} and show that the limit orbit $Gy$ is always $\delta$-geodesic. In Section \ref{sec_split}, we study the quantitative splitting behavior of $Gy$ in the associated family $\{(sY,y,G)\}_{s>0}$. In Section \ref{sec_eu}, we prove Theorem \ref{almost_eu_orbit} and Theorem A. 

\tableofcontents

\section{Preliminaries}\label{sec_pre}

\noindent\textbf{1.1 Almost splitting}

Cheeger and Colding proved a quantitative splitting result for manifolds with almost nonnegative Ricci curvature \cite{CC1}. Here we need a version for Ricci limit spaces, which follows directly from the result on manifolds. We denote $\mathcal{M}(n,0)$ the set of all Ricci limit spaces coming from some sequence of complete Riemannian $n$-manifolds $(M_i,p_i)$ of $\mathrm{Ric}\ge 0$. Given $y_-,y_+$ in a space $Y\in\mathcal{M}(n,0)$, recall that the excess function is
$$e_{y_+,y_-}(z)=d(z,y_+)+d(z,y_-)-d(y_+,y_-).$$

\begin{thm}\label{thm_CC_split}\cite{CC1}
	Let $(Y,y)\in\mathcal{M}(n,0)$. Let $y_-,y_+\in Y$ with $d(y_\pm,y)=L\gg R$ and $e_{y_-,y_+}(y)\le \delta$, then there exists a length space $(X,x)$ such that
	$$d_{GH}(B_R(y),B_R(0,x))\le \Phi(\delta,L^{-1}|n,R),$$
	where $(0,x)$ is a point in the metric product $\mathbb{R}\times X$.
\end{thm}

Here $\Phi(\delta_1,...,\delta_k|c_1,...,c_l)$ means a nonnegative function depending on $\delta_1,...,\delta_k$ and $c_1,...,c_l$ such that
$$\lim\limits_{\delta_1,...,\delta_k\to 0} \Phi(\delta_1,...,\delta_k|c_1,...,c_l)=0.$$

We briefly recall how Theorem \ref{thm_CC_split} is proved for manifolds since we need some elements from this proof later. Let $(M,y)$ be a complete Riemannian manifold of $\mathrm{Ric}\ge0$. Given $y_-,y_+\in M$ with the assumptions in Theorem \ref{thm_CC_split}, the partial Busemann function $b$ is defined as $$b(z)=d(z,y_+)-d(y,y_+).$$
Let $h$ be the solution to the Dirichlet problem
$$\begin{cases} \Delta h=0 \quad & \text{on } B_{20R}(y),\\
h=b \quad & \text{on } \partial B_{20R}(y). \end{cases}$$
Among other properties, $h$ satisfies the following estimates.
\begin{prop}\cite{CC1}\label{split_estimates}
	Let $z\in B_R(y)$ and $w\in h^{-1}(h(y))$ be a closest point from $z$ to $h^{-1}(h(y))$. Then the followings hold:\\
	(1) ($C^0$-estimate) $|h(z)-b(z)|\le \Phi(\delta,L^{-1}|n,R)$;\\
	(2) (Almost parallel) $d(z,w)=|h(z)-h(y)|\pm \Phi(\delta,L^{-1}|n,R)$;\\
	(3) (Almost Pythagorean) $|d(y,w)^2+d(z,w)^2-d(z,y)^2|\le \Phi(\delta,L^{-1}|n,R)$.
\end{prop}
Then the map
$$F:B_R(y)\to \mathbb{R}\times h^{-1}(h(y)),\quad z\mapsto (h(z)-h(y),w),$$
where $w\in h^{-1}(h(y))$ is a closest point from $z$ to $h^{-1}(h(y))$, is a $\Phi(\delta,L^{-1}|n,R)$-approximation between $B_R(y)$ and $B_R(0,x)$ \cite{CC1}.\\

\noindent\textbf{1.2 Asymptotic geometry}

Let $(M,p)$ be an open manifold of $\mathrm{Ric}\ge 0$ and let $r_i\to\infty$. Then there is a subsequence converging in the pointed Gromov-Hausdorff topology:
$$(r_i^{-1}M,p)\overset{GH}\longrightarrow (Z,z).$$
The limit space $(Z,z)$ is called an \textit{asymptotic cone} of $M$. $(Z,z)$ in general depends on the scaling sequence $r_i$, so $M$ may not have a unique asymptotic cone. 

Let $(\widetilde{M},\tilde{p})$ be the Riemannian universal cover of $(M,p)$ and let $\Gamma$ be the fundamental group $\pi_1(M,p)$, which acts on $\widetilde{M}$ isometrically. For a sequence $r_i\to \infty$, we can consider a convergent subsequence in the pointed equivariant Gromov-Hausdorff topology \cite{FY}:
$$(r_i^{-1}\widetilde{M},\tilde{p},\Gamma)\overset{GH}\longrightarrow (Y,y,G),$$
where $G$ is a closed subgroup of the isometry group of $Y$. We call $(Y,y,G)$ an \textit{equivariant asymptotic cone} of $(\widetilde{M},\Gamma)$. Note that $(Y,y)\in\mathcal{M}(n,0)$, so Theorem \ref{thm_CC_split} can be applied. According to \cite{CC2,CN}, the limit group $G$ is a Lie group. In particular, $G/G_0$ is discrete, where $G_0$ is the identity component subgroup of $G$. Hence there is a positive distance between different components of the orbit $Gy$.

Let $\Omega(\widetilde{M},\Gamma)$ be the set of all equivariant asymptotic cones of $(\widetilde{M},\Gamma)$. The result below is well-known (see \cite[Proposition 2.1]{Pan_eu} for a proof).

\begin{prop}\label{cpt_cnt}
	Let $(M,p)$ be an open $n$-manifold with $\mathrm{Ric}\ge 0$. Then the set $\Omega(\widetilde{M},\Gamma)$ is compact and connected in the pointed equivariant Gromov-Hausdorff topology.
\end{prop}

\noindent\textbf{1.3 Gromov-Hausdorff convergence of closed subsets}

To study the quantitative splitting behavior of the limit orbit $Gy$ in an equivariant asymptotic cone $(Y,y,G)$, we need a natural notion to measure the closeness between closed subsets in two nearby pointed spaces.

\begin{defn}\label{def_GH_subset}
	For $i=1,2$, let $(X_i,x_i)$ be a complete pointed length space and let $A_i$ be a closed subset of $X_i$ with $x_i\in A_i$. Let $\delta>0$. We say that a map $\phi:X_1\to X_2$ is an $\delta$-approximation from $(X_1,x_1,A_1)$ to $(X_2,x_2,A_2)$ if\\
	(1) $d(x_2,\phi(x_1))\le\delta$,\\
	(2) $|d(z,z')-d(\phi(z),\phi(z'))|\le\delta$ for all $z,z'\in B_{1/\delta}(x_1)$,\\
	(3) $B_{1/\delta}(x_2)\subseteq B_{\delta}\left(\phi(B_{1/\delta}(x_1))\right)$,\\
	%(4) $\phi(B_{1/\delta}(x_1))$ is $\delta$-dense in $B_{1/\delta}(x_2)$, \\
	(4) $B_{1/\delta}(x_2)\cap A_2 \subseteq B_{\delta}\left(\phi(B_{1/\delta}(x_1)\cap A_1)\right)$.
	%(6) $\phi(B_{1/\delta}(x_1)\cap A_1)$ is $\delta$-dense in $B_{1/\delta}(x_2)\cap A_2$.
	
	We say that 
	$$d_{GH}((X_1,x_1,A_1),(X_2,x_2,A_2))\le\delta,$$
	if there are $\delta$-approximation maps
	$$\phi:(X_1,x_1,A_1)\to (X_2,x_2,A_2), \quad \psi:(X_2,x_2,A_2)\to (X_1,x_1,A_1).$$
\end{defn}

Note that the first three conditions in Definition \ref{def_GH_subset} say that $\phi:X_1\to X_2$ is an $\delta$-approximation from $(X_1,x_1)$ to $(X_2,x_2)$ in the pointed Gromov-Hausdorff sense. Together with the last condition, the closure of ${\phi(A_1)}$ and $A_2$, as (pointed) closed subsets of $X_2$, are $\delta$-close in the pointed Hausdorff sense.

\begin{rem}
	In practice, we only need to check one side of the approximations in Definition \ref{def_GH_subset}. Given an $\delta$-approximation $\phi:(X_1,x_1,A_1)\to (X_2,x_2,A_2)$, then one can construct a $3\delta$-approximation $\psi:(X_2,x_2,A_2)\to (X_1,x_1,A_1)$: for any $z_2\in B_{1/\delta}(x_2)$, by (3) there is $z_1\in B_{1/\delta}(x_1)$ such that $d(\phi(z_1),z_2)\le\delta$, then assign $\psi(z_2)$ as $z_1$.
\end{rem}

\begin{rem}
	Though $d_{GH}$ as in Definition \ref{def_GH_subset} may not satisfy the triangle inequality, a weaker estimate always holds:
	\begin{align*}
		&\ d_{GH}((X_1,x_1,A_1),(X_3,x_3,A_3))\\
		\le&\ 2d_{GH}((X_1,x_1,A_1),(X_2,x_2,A_2))+2d_{GH}((X_2,x_2,A_2),(X_3,x_3,A_3)).
	\end{align*}
\end{rem}

\begin{rem}\label{rem_subset_precpt}
For a sequence $(X_i,x_i)$ converging to $(X,x)$ with a sequence of closed subsets $A_i\subseteq X_i$ containing $x_i$, we can always extract a subsequence so that $$(X_{i(j)},x_{i(j)},A_{i(j)})\overset{GH}\longrightarrow (X,x,A)$$ 
in the sense of Definition \ref{def_GH_subset}, where $A$ is some closed subset of $X$ containing $x$. To see this precompactness result, one can consider the closure of $\phi_i(A_i)$ in $X_i$, where $\phi_i:X_i\to X$ are $\delta_i$-approximation maps with $\delta_i\to 0$, then the sequence $\{(\overline{\phi_i(A_i)},\phi(x_i))\}$ subconverges in the pointed Hausdorff sense to some limit $(A,x)$.
\end{rem}

\begin{rem}
	For a pointed equivariant Gromov-Hausdorff convergent sequence
	$$(X_i,x_i,G_i)\overset{GH}\longrightarrow (X,x,G),$$
	by definition we have corresponding convergence of orbits $G_ix_i$ to $Gx$ in the sense of Definition \ref{def_GH_subset}:
	$$(X_i,x_i,G_ix_i)\overset{GH}\longrightarrow (X,x,Gx).$$
\end{rem}

\section{Almost geodesic limit orbits}\label{sec_geo}
Through this section and beyond, we assume that $(M,p)$ is an open $n$-manifold with nonnegative Ricci curvature and an infinite fundamental group $\Gamma=\pi_1(M,p)$; we also always assume that $E(M,p)\le\epsilon$ unless otherwise noted, where $\epsilon>0$ is a small number that will be determined through the text. To avoid ambiguities, symbol $\epsilon$ will be exclusively used for this purpose. $\delta_\epsilon$ means a constant only depending on $\epsilon$ with $\lim_{\epsilon\to 0} \delta_\epsilon=0$.

In this section, we prove that in any $(Y,y,G)\in\Omega(\widetilde{M},\Gamma)$, the orbit $Gy$ is $\delta_\epsilon$-geodesic (Proposition \ref{limit_delta_geodesic}). We also prove some properties for $\delta_\epsilon$-geodesic subsets for later use.

For a closed subset $N$ in a length space $X$, we say that $N$ is \textit{geodesic} in $X$, if the extrinsic and intrinsic metrics on $N$ agree. To quantify this, we introduce the notion \textit{$\delta$-geodesic}.

\begin{defn}\label{def_delta_geo}
	Let $\delta>0$. Let $X$ be a length space and let $N$ be a closed subset of $X$. We say that $N$ is \textit{$\delta$-geodesic in $X$}, if for any two points $a,b\in N$, there is a chain of points $z_0=a,...,z_j,...,z_k=b$ in $N$ such that
	$$\sum_{j=1}^{k} d(z_{j-1},z_{j})\le (1+\delta)\cdot d(a,b)\quad \text{and}\quad d(z_{j-1},z_{j})\le\delta\cdot d(a,b) \text{ for all } j=1,...,k.$$
\end{defn}

Note that being $\delta$-geodesic is scaling invariant.

Since $\Gamma$ acts freely on $\widetilde{M}$, we can assign a natural metric $\rho$ on $\Gamma$ by
$$\rho(\gamma,\gamma')=d(\gamma\tilde{p},\gamma'\tilde{p}).$$

\begin{lem}\label{pre_limit_delta_geodesic}
	For any $\gamma\in\Gamma$ with $\rho(e,\gamma)$ sufficiently large, there are $\gamma_1,...,\gamma_k\in\Gamma$ such that the following holds:\\
	(1) $\prod_{j=1}^k \gamma_j=\gamma$,\\
	(2) $\sum_{j=1}^k \rho(e,\gamma_j)\le (1+\delta_\epsilon)\cdot \rho(e,\gamma)$,\\
	(3) $\rho(e,\gamma_j)\le \delta_\epsilon \cdot \rho(e,\gamma)$ for all $j=1,...,k$,\\
	where $\delta_\epsilon\to 0$ as $\epsilon\to 0$.
\end{lem}

\begin{proof}
	The proof is similar to that of \cite[Proposition 2.2]{Pan_escape}.
	
	We put $\Gamma(R)=\{\gamma\in \Gamma | d(\tilde{p},\gamma\tilde{p})\le R \}$ and $$D(R)=\max_{\gamma\in \Gamma(R)}d_H(x,c_\gamma).$$ It follows from $E(M,p)\le\epsilon$ that 
	$$\limsup_{R\to\infty} \dfrac{D(R)}{R}\le\epsilon.$$
	For $\eta>0$, which we will determine later, we define $s(\eta,R)=2(\eta^{-1}+1)\cdot D(R)$. 
	
	Let $\gamma\in\Gamma$ with $R=\rho(e,\gamma)$ and let $c$ be a representing geodesic loop of $\gamma$. It is clear that $c$ is contained in $\overline{B_{D(R)}}(p)$. Let $\tilde{c}$ be the lift of $c$ starting at $\tilde{p}$; we can assume that $\tilde{c}:[0,R]\to\widetilde{M}$ is of unit speed. Following the same argument as in the proof of Proposition 2.2 in \cite{Pan_escape}, we can choose a series of points $\{\tilde{c}(t_j)\}_{j=1}^k$ such that $t_0=0$, $t_k=R$ and
	$$t_j-t_{j-1}=2\eta^{-1}D(R) \text{ for all }j=1,...,k-1,\quad t_k-t_{k-1}\le 2\eta^{-1}D(R).$$
	We also choose $\beta_0=e$, $\beta_k=\gamma$, and $\beta_j\in\Gamma$ such that $d(\beta_j\tilde{p},\tilde{c}(t_j))\le D(R)$ for $j=1,...,k-1$. Then $\{\gamma_j=\beta^{-1}_{j-1}\beta_j\}_{j=1}^k$ satisfies $\prod_{j=1}^k \gamma_j=\gamma$,
	$$\sum_{j=1}^k \rho(e,\gamma_j)\le (1+\eta)\rho(e,\gamma),$$
	and 
	$$\rho(e,\gamma_j)\le s(\eta,R)\le 4(\eta^{-1}+1)\epsilon\cdot \rho(e,\gamma)$$
	when $R$ is sufficiently large. Setting $\eta=\sqrt{\epsilon}$ and $\delta_\epsilon=4(\sqrt{\epsilon}+\epsilon)$, we complete the proof.
\end{proof}

\begin{prop}\label{limit_delta_geodesic}
	For any $(Y,y,G)\in \Omega(\widetilde{M},\Gamma)$, the orbit $Gy$ is $\delta_\epsilon$-geodesic in $Y$, where $\delta_\epsilon\to 0$ as $\epsilon\to 0$.
\end{prop}

\begin{proof}
	Let $r_i\to\infty$ such that
	$$(r_i^{-1}\widetilde{M},\tilde{p},\Gamma)\overset{GH}\longrightarrow (Y,y,G).$$
	Let $g\in G$ with $gy\not= y$ and let $\gamma_i\in \Gamma$ converging to $g$ associated to the above convergence. We put $R_i=d(\gamma_i\tilde{p},\tilde{p})\to\infty$. It follows from Lemma \ref{pre_limit_delta_geodesic} that for any $i$ large, we can find a word $\prod_{j=1}^{k_i} \gamma_{i,j}=\gamma_i$ with
	$$\sum_{j=1}^{k_i} \rho(e,\gamma_{i,j})\le (1+\delta_\epsilon)\cdot R_i,\quad \rho(e,\gamma_{i,j})\le \delta_\epsilon\cdot R_i.$$
	We group successive portions of the word $\prod_{j=1}^{k_i} \gamma_{i,j}$ into a new word $\prod_{j=1}^{K_i} g_{i,j}=\gamma_i$ such that
	$$\delta_\epsilon R_i\le\rho(e,g_{i,j})\le 2\delta_\epsilon R_i.$$
	It is clear that
	$$\sum_{j=1}^{K_i}\rho(e,g_{i,j})\le (1+\delta_\epsilon)\cdot R_i,\quad K_i\le (1+\delta_\epsilon)/\delta_\epsilon.$$
	Passing to a subsequence, we assume that all $K_i$ are the same and each sequence $\{g_{i,j}\}_i$ converges to $g_j\in G$ associated to $(r_i^{-1}\widetilde{M},\tilde{p},\Gamma)\overset{GH}\longrightarrow (Y,y,G)$. Then 
	$$d_Y(g_jy,y)=\lim\limits_{i\to\infty} r_i^{-1}\rho(e,g_{i,j})\le r_i^{-1}\cdot 2\delta_\epsilon R_i=2\delta_\epsilon\cdot d_Y(gy,y),$$
	$$\sum_{j=1}^K d_Y(g_jy,y)=\lim\limits_{i\to\infty} \sum_{j=1}^K r_i^{-1}\rho(e,g_{i,j})\le\lim\limits_{i\to\infty} r_i^{-1}(1+\delta_\epsilon)R_i=(1+\delta_\epsilon)d_Y(gy,y).$$
	This shows that the limit orbit $Gy$ is $2\delta_\epsilon$-geodesic in $Y$.
\end{proof}

\begin{cor}\label{cor_cnt_orbit}
	Given that $\epsilon>0$ is sufficiently small, the orbit $Gy$ is connected for any $(Y,y,G)\in\Omega(\widetilde{M},\Gamma)$.
\end{cor}

\begin{proof}
	We choose $\epsilon>0$ small so that $\delta_\epsilon\le 0.1$, where $\delta_\epsilon$ is the constant in Proposition \ref{limit_delta_geodesic}.
	
	We argue by contradiction. Suppose that the orbit $Gy$ is not connected for some $(Y,y,G)\in\Omega(\widetilde{M},\Gamma)$. Let $C$ be the connected component of $Gy$ containing $y$. We choose an orbit point $hy\in Gy-C$ such that
	$$d(hy,y)=\min_{z\in Gy-C}d(z,y)=\min_{z\in Gy-C} d(z,C).$$ 
	By Proposition \ref{limit_delta_geodesic}, there is a chain of orbit points $z_0=y,...,z_j,...,z_k=hy$ such that $d(z_{j-1},z_j)\le 0.1\cdot d(hy,y)$. 
	
	We claim that $z_1\in C$. Otherwise, we have $$0<d(C,z_1)\le d(z_0,z_1)\le 0.1\cdot d(hy,y);$$  
	this cannot happen due to our choice of $hy$.
	
	Inductively, we have $z_j\in C$ for all $j$. A contradiction.
\end{proof}

For the rest of this section, we prove some lemmas on convergence of $\delta$-geodesic subsets, which will be used in the next section.

\begin{lem}\label{to_geodesic}
	Let $(Y_i,y_i)$ be a sequence of complete pointed length spaces and $N_i$ be a closed subset of $Y_i$ containing $y_i$ for each $i$. Suppose that
	$$(Y_i,y_i,N_i)\overset{GH}\longrightarrow (Y,y,N)$$
	and $N_i$ is $\delta_i$-geodesic in $Y_i$, where $\delta_i\to 0$. Then $N$ is geodesic in $Y$.
\end{lem}

\begin{proof}
	Let $a,b\in N$. From the convergence, we can choose points $a_i$ and $b_i$ in $N_i$ converging to $a$ and $b$ respectively. Because each $N_i$ is $\delta_i$-geodesic in $Y_i$, there is a series of points $z_{i,0}=a_i,...,z_{i,j},...,z_{i,k(i)}=b_i$ in $N_i$ such that
	$$\sum_{j=1}^{k(i)} d(z_{i,j-1},z_{i,j})\le (1+\delta_i) d(a_i,b_i),\quad d(z_{i,j-1},z_{i,j})\le\delta_id(a_i,b_i)\to 0.$$
	For any $\eta>0$, we choose a series of points $\{w_{i,j}\}_{j=0}^{K(i)}$ from $\{z_{i,j}\}_{j=0}^{k(i)}$ such that $w_{i,0}=z_{i,0}$, $w_{i,K(i)}=z_{i,k(i)}$, and
	$$\eta/2\le d(w_{i,j-1},w_{i,j})\le \eta$$
	for all $j=1,...,K(i)$. Note that 
	$$\sum_{j=0}^{K(i)} d(w_{i,j-1},w_{i,j})\le (1+\delta_i)d(a_i,b_i),\quad K(i)\le (1+\delta_i)d(a_i,b_i)/(\eta/2)\to 2d(a,b)/\eta.$$
	Passing to a subsequence, we assume that all $K(i)$ are the same $K$ and each sequence $\{w_{i,j}\}_i$ converges to $w_j\in N$. Then $\{w_j\}_{j=1}^K\subseteq N$ satisfies
	$$\sum_{j=0}^{K}d(w_{j-1},w_{j})\le d(a,b),\quad d(w_{j-1},w_{j})\le\eta \text{ for all } j=1,...,K.$$
	This shows that $N$ is geodesic in $Y$. 
\end{proof}
 
The following result in \cite[Lemma 3.1]{Pan_escape} characterizes any closed and geodesic subset in a metric product $\mathbb{R}^k\times X$ that contains a slice of $\mathbb{R}^k$. 

\begin{lem}\label{geo_in_product}
	Let $X$ be a locally compact length metric space. Let $N$ be a closed subset in the product metric space $\mathbb{R}^k\times X$, where $\mathbb{R}^k$ is endowed with the standard Euclidean metric. Suppose that $N$ is geodesic in $\mathbb{R}^k\times X$ and $N$ contains a slice $\mathbb{R}^k\times \{x\}$ for some $x\in X$. Then $N$ equals $\mathbb{R}^k\times Z$ as a subset of $\mathbb{R}^k\times X$, where
	$$Z=N \cap (\{0\}\times X);$$
	in particular, $N$ is a product metric.
\end{lem}

We prove a quantitative version of Lemma \ref{geo_in_product} in terms of $\delta$-geodesic subsets and Gromov-Hausdorff closeness.  
 
\begin{lem}\label{subset_almost_product}
	Let $(X,x)$ be a pointed locally compact length space and let $(Y,y)\in\mathcal{M}(n,0)$. Let $N$ be a closed and $\delta$-geodesic subset in $Y$ with $y\in N$. Suppose that
	$N$ has a closed subset $S$ such that
	$$d_{GH}((Y,y,S),(\mathbb{R}^k\times X,(0,x),\mathbb{R}^k\times\{x\}))\le\delta.$$
	Then there is a length space $(X',x')$, which is possibly different from $(X,x)$, and a closed geodesic subset $Z'\subseteq X'$ such that
	$$d_{GH}((Y,y,N),(\mathbb{R}^k\times X',(0,x'),\mathbb{R}^k\times Z'))\le \Phi(\delta|n).$$
\end{lem}

\begin{proof}
	Let $(Y_i,y_i)$ be any sequence of spaces in $\mathcal{M}(n,0)$ satisfying\\
	(1) $Y_i$ has a closed and $\delta_i$-geodesic subset $N_i$ containing $y_i$,\\
	(2) $N_i$ has a closed subset $S_i$ with
	$$d_{GH}((Y_i,y_i,S_i),(\mathbb{R}^k\times X_i,(0,x_i),\mathbb{R}^k\times\{x_i\}))\le\delta_i\to 0.$$
	Passing to a subsequence, the sequences $(Y_i,y_i,S_i)$ and $(\mathbb{R}^k\times X_i,(0,x_i),\mathbb{R}^k\times\{x_i\})$ converge to the same limit space $$(Y_\infty,y_\infty,S_\infty)=(\mathbb{R}^k\times X_\infty,(0,x_\infty),\mathbb{R}^k\times\{x_\infty\}).$$ 
	Due to Remark \ref{rem_subset_precpt}, we can also assume that
	$$(Y_i,y_i,N_i)\overset{GH}\longrightarrow (Y_\infty,y_\infty,N_\infty).$$
	By Lemma \ref{to_geodesic}, $N_\infty$ is geodesic in $Y_\infty=\mathbb{R}^k\times X_\infty$; moreover, $N_\infty$ contains a slice $\mathbb{R}^k\times \{x_\infty\}$. It follows from \ref{geo_in_product} that $N_\infty$, as a subset of $\mathbb{R}^k\times X_\infty$, is a product $\mathbb{R}^k\times Z_\infty$, where $Z_\infty=N_\infty\cap (\{0\}\times X_\infty)$. Finally, note that $Z_\infty$ is geodesic in $X_\infty$ and
	$$d_{GH}((Y_i,y_i,N_i),(\mathbb{R}^k\times X_\infty,(0,x_\infty),\mathbb{R}^k\times Z_\infty))\to 0.$$
\end{proof}

\section{Almost splitting of limit orbits under all scales}\label{sec_split}

We study the quantitative splitting of the orbit $Gy$ for any $(Y,y,G)\in\Omega(\widetilde{M},\Gamma)$. Again, we assume that $E(M,p)\le\epsilon$ unless otherwise noted.

\begin{lem}\label{non_cpt_orbit}
	Let $(M,p)$ be an open $n$-manifold with $\mathrm{Ric}\ge0$, an infinite fundamental group, and $E(M,p)<1/2$, then the orbit $Gy$ is non-compact for any $(Y,y,G)\in\Omega(\widetilde{M},\Gamma)$.
\end{lem}

\begin{proof}
	Suppose that $(Y,y,G)\in\Omega(\widetilde{M},\Gamma)$ has a compact orbit $Gy$. Let $r_i\to\infty$ such that
	$$(r_i^{-1}\widetilde{M},\tilde{p},\Gamma)\overset{GH}\longrightarrow (Y,y,G).$$
	Let $D>0$ such that $Gy\subseteq B_D(y)$. 
	
	Since $\Gamma$ is infinite, we can find a sequence $\gamma_i\in \Gamma$ such that $d(\gamma_i\tilde{p},\tilde{p})\ge 10r_iD$. %Note that this implies that $r_i^{-1}d(\gamma_i\tilde{p},\tilde{p})\to\infty$. Otherwise, we can consider the convergence
	%$$(r_i^{-1}\widetilde{M},\tilde{p},\Gamma,\gamma_i)\overset{GH}\longrightarrow (Y,y,G,g)$$
	%with $g\in G$ and $d(gy,y)\ge 10D$, but all orbit points should be contained in $B_D(y)$.
	Because $E(M,p)<1/2$, by \cite[Lemma 2.1]{Pan_escape}, $\Gamma$ is finitely generated. Let $S$ be a finite generating set and let $R=\max_{\gamma\in S} d(\gamma\tilde{p},\tilde{p})$. For each $\gamma_i$, we write 
	$\gamma_i=\prod_{k=1}^{N_i} g_{i,k}$, where $g_{i,k}\in S$. Let $h_{i,j}=\prod_{k=1}^j g_{i,k}$ for each $j$. Then $\{h_{i,j}\tilde{p}\}_{j=1}^{N_i}$ is a series of orbit points starting from $\tilde{p}$ and ending at $\gamma_i\tilde{p}$ such that
	$$d(h_{i,j}\tilde{p},h_{i,j+1}\tilde{p})\le R$$
	for each $j$. In particular, for each $i$ we can find some $h_{i,j(i)}\tilde{p}$ such that
	$$2r_iD-R\le d(h_{i,j(i)}\tilde{p},\tilde{p})\le 2r_iD+R.$$
	Then 
	$$(r_i^{-1}\widetilde{M},\tilde{p},\Gamma,h_{i,j(i)})\overset{GH}\longrightarrow (Y,y,G,h)$$
	with $h\in G$ and $d(hy,y)=2D$. This contradicts the hypothesis that $Gy$ is contained in $B_D(y)$.
\end{proof}

With Lemma \ref{non_cpt_orbit}, we first show that the asymptotic cone $Y$ of $\widetilde{M}$ almost splits off a line when $E(M,p)\le\epsilon$.

\begin{lem}\label{noncpt_space_split}
    Let $(Y,y,G)\in\Omega(\widetilde{M},\Gamma)$. Then there is a length space $X$ such that
	$$d_{GH}((Y,y),(\mathbb{R}\times X, (0,x)))\le \Phi(\epsilon|n).$$
\end{lem}

\begin{proof}
	For any small $\eta>0$, we will determine an $\epsilon>0$ so that $$d_{GH}((Y,y),(\mathbb{R}\times X, (0,x)))\le \eta$$
	when $E(M,p)\le \epsilon$. 
	
	Since $Gy$ is non-compact from Lemma \ref{non_cpt_orbit}, we can choose a point $gy$ so that $d(y,gy)=L\gg R$, where $L$ and $R$ will be determined later. We know that $Gy$ is $\delta_\epsilon$-geodesic by Proposition \ref{limit_delta_geodesic}. This provides a series of points $z_0=y,...,z_j,...,z_k=gy$ in $Gy$ such that
	$$\sum_{j=0}^{k} d(z_{j-1},z_{j})\le (1+\delta_\epsilon)L,\quad d(z_{j-1},z_{j})\le \delta_\epsilon L \text{ for all }j.$$
	We choose a point $w\in \{z_j\}_j$ that is about the middle between $y$ and $gy$; more precisely, $w$ such that
	$$(1/2-\delta_\epsilon)L\le d(y,w)\le (1/2+\delta_\epsilon)L.$$
	We write this $w=hy$ for some $h\in G$. Then for $y_-=h^{-1}y$ and $y_+=h^{-1}gy$, we have
	$$e_{y_-,y_+}(y)\le 2\delta_\epsilon L.$$
	By quantitative splitting Theorem \ref{thm_CC_split}, we see that
	$$d_{GH}(B_R(y),B_R(0,x))\le \Phi(\delta_\epsilon L,L^{-1}|n,R),$$
	where $(0,x)\in \mathbb{R}\times X$ for some length space $X$.
	
	Now we set $R=2\eta^{-1}$, $L=1/\sqrt{\delta_\epsilon}$, and $\epsilon>0$ sufficiently small so that
	$$\Phi(\delta_\epsilon L,L^{-1}|n,R)=\Phi(\epsilon|n,2\eta^{-1})\le\eta.$$
	With this $\epsilon$, we have 
	$$d_{GH}((Y,y),(\mathbb{R}\times X,(0,x)))\le\eta.$$
\end{proof}

\begin{lem}\label{orbit_close_to_segment}
	In the context of Theorem \ref{thm_CC_split}, let $S=\{z_j\}_{j=0}^k$ be a series of points in $Y$ such that\\
	(1) $z_0=y_-$, $z_k=y_+$, and $y\in S$;\\
	(2) $\sum_{j=1}^{k} d(z_{j-1},z_{j})\le (1+\delta)L$;\\
	(3) $d(z_{j-1},z_{j})\le \delta L$ for all $j=1,...,k$.\\
	Then 
	$$d_{GH}((B_R(y),y,S\cap B_R(y)),(B_R (0,x),(0,x),[-R,R]\times \{x\})\le \Phi(L^{-1},\delta L|n,R).$$
\end{lem}

\begin{proof}
	It suffices to prove the statement for manifolds.
	
	Let $$F:B_R(y)\to \mathbb{R}\times h^{-1}(h(y)),\quad z\mapsto (h(z)-h(y),w)$$ be the Gromov-Hausdorff approximation mentioned in Section \ref{sec_pre}. We first show that
	$$d(w_j,y)\le \Phi(\delta,L^{-1}|n,R)$$
	for any point $z_j\in B_R(y)\cap S$, where $w_j\in h^{-1}(h(y))$ is a closed point from $z_j$ to $h^{-1}(h(y))$. Since $y\in S$, we write $y=z_J$ for some $J=1,...,k$. We consider the case that $j<J$; the case $j>J$ would be similar. By Proposition \ref{split_estimates}(1,2) and condition (2), we have
	\begin{align*}
		d(z_j,w_j) &= h(z_j)-h(y)\pm \Phi(\delta,L^{-1}|n,R)\\
		&= b(z_j)-b(y)\pm \Phi(\delta,L^{-1}|n,R)\\
		&= d(y_+,z_j)-d(y_+,y)\pm \Phi(\delta,L^{-1}|n,R)\\
		&= \sum_{i=j}^{k-1}d(z_i,z_{i+1})-\sum_{i=J}^{k-1}d(z_i,z_{i+1})\pm \Phi(\delta,L^{-1}|n,R)\\
		&= \sum_{i=i}^{J-1} d(z_i,z_{i+1})\pm \Phi(\delta,L^{-1}|n,R)\\
		&= d(z_j,y)\pm \Phi(\delta,L^{-1}|n,R).
	\end{align*}
    Together with Proposition \ref{split_estimates}(3), we see that
    $$d(w_j,y)\le \Phi(\delta,L^{-1}|n,R).$$
    %Since $F(z_j)=(h(z_j)-h(y),w_j)$, this shows that 
    %$$F(B_R(y)\cap S)\subseteq B_{\Phi}([-R,R]\times\{x\}).$$ 
    
    Let $(t,x)\in [-R,R]\times \{x\}$. By condition (3), we can choose $z_j\in B_R(y)\cap S$ such that 
    %$$d(z_j,y)\in [|t|-\delta L,|t|+\delta L].$$
    %Thus
    \begin{align*}
    h(z_j)-h(y)&= \pm d(z_j,y)\pm \Phi(\delta,L^{-1}|n,R)\\
    &\in[t-\delta L-\Phi(\delta,L^{-1}|n,R),t+\delta L+\Phi(\delta,L^{-1}|n,R)].
    \end{align*}
    Then this $z_j$ satisfies that
    \begin{align*}
    d(F(z_j),(t,x))^2&=|h(z_j)-h(y)-t|^2+d(w_j,y)^2\\
    &\le \Phi(L^{-1},\delta L|n,R)^2+\Phi(\delta,L^{-1}|n,R)^2.
    \end{align*}
    This shows that
    $$[-R,R]\times \{x\}\subseteq B_{\Phi(L^{-1},\delta L|n,R)}\left(F(B_R(y)\cap S)\right)$$
    and the result follows.
\end{proof}

%With all the preparations above, Theorem \ref{noncpt_orbit_split} follows from Lemmas \ref{subset_almost_product}, \ref{noncpt_space_split}, and \ref{orbit_close_to_segment}.

With Lemmas \ref{subset_almost_product} and \ref{orbit_close_to_segment}, the set $\{z_j\}\subseteq Gy$ that we used in Lemma \ref{noncpt_space_split} shows the almost splitting of $Gy$.

\begin{prop}\label{noncpt_orbit_split}
	Let $(Y,y,G)\in\Omega(\widetilde{M},\Gamma)$. Then there are a length space $X$ and a closed geodesic subset $Z\subseteq X$ such that
	$$d_{GH}((Y,y,Gy),(\mathbb{R}\times X, (0,x),\mathbb{R}\times Z))\le \Phi(\epsilon|n).$$
\end{prop}

\begin{proof}
	Let $\eta>0$. We continue to use the notations in the proof of Lemma \ref{noncpt_space_split}, from which we see that the points $y_-$ and $y_+$ with $$e_{y_-,y_+}(y)\le2\delta_\epsilon L=2\sqrt{\delta_\epsilon}$$ provides a $\Phi(\epsilon|n,2\eta^{-1})$-approximation between $(Y,y)$ and $(\mathbb{R}\times X,(0,x))$. Moreover, we also have a series of points $S=\{z'_j=h^{-1}z_j\}_{j=0}^{k}\subseteq Gy$ with\\
	(1) $z'_0=y_-$, $z'_k=y_+$, and $z'_j=y$ for some $j$;\\
	(2) $\sum_{j=1}^{k-1} d(z'_{j-1},z'_{j})\le (1+\delta_\epsilon)L$;\\
	(3) $d(z'_{j-1},z'_{j})\le \delta_\epsilon L$ for all $j=1,...,k$.\\
	It follows from Lemma \ref{orbit_close_to_segment} that
	$$d_{GH}((B_R(y),y,S\cap B_R(y)),(B_R (0,x),(0,x),[-R,R]\times \{x\})\le \Phi(\epsilon|n,R).$$
	Recall that $R$ is chosen as $2\eta^{-1}$, then for $\epsilon$ small, we have
	$$d_{GH}((Y,y,S),(\mathbb{R}\times X,(0,x),\mathbb{R}\times\{x\}))\le \Phi(\epsilon|n,2\eta^{-1})\le\eta.$$
	Then the result follows from Lemma \ref{subset_almost_product}.
\end{proof}	
	
When $E(M,p)=0$, \cite[Lemma 3.2]{Pan_escape} shows that the orbit $Gy$ is a metric product $\mathbb{R}^k\times Z$, where $Z$ is compact. With Proposition \ref{noncpt_orbit_split}, where we showed that $Gy$ is $\Phi(\epsilon|n)$-close to a metric product $\mathbb{R}\times Z$, we wish to continue the splitting process if $Z$ is non-compact.

As mentioned in the introduction, the main issue here is that (pointed) Gromov-Hausdorff closeness in general cannot measure whether a subset is compact or not. In the context of 
$$d_{GH}((Y,y,Gy),(\mathbb{R}\times X, (0,x),\mathbb{R}\times Z))\le \Phi(\epsilon|n),$$
if $\mathrm{diam}(Z)\gg \Phi(\epsilon|n)^{-1}$, then whether $Z$ is compact or not does not a make difference. % also, if $\mathrm{diam}(Z)\le \Phi(\epsilon|n)^{-1}$, then $Z$ may be chosen as a point $\{x\}$ with a slightly larger $\Phi(\epsilon|n)$.
To overcome this, for each $(Y,y,G)\in\Omega(\widetilde{M},\Gamma)$, we shall consider a corresponding family of spaces $\{(sY,y,G)\}_{s>0}\subseteq \Omega(\widetilde{M},\Gamma)$. 

Our main goal for the rest of this section is the following result.
\begin{prop}\label{split_all_scales}
	%Let $(M,p)$ be an open manifold of $\mathrm{Ric}\ge 0$ and $E(M,p)\le\epsilon$. 
	Let $(Y,y,G)\in\Omega(\widetilde{M},\Gamma)$. Then there is an integer $k$ such that for all $s>0$,
	$$d_{GH}((sY,y,Gy),(\mathbb{R}^k\times X_s,(0,x_s),\mathbb{R}^k\times Z_s))\le \Phi(\epsilon|n),$$
	where $(X_s,x_s)$ is a length space depending on $sY$, and $Z_s$ is a closed geodesic subset of $X_s$. Moreover, one of the following holds:\\
	(1) $Z_s$ is a single point for all $s>0$;\\
	(2) $\mathrm{diam}(Z_s)\in[0.9,1.1]$ for some $s>0$.
\end{prop}

In next section, we will further rule out case (2) above.

\begin{lem}\label{large_diam_split}
	Let $(Y,y)\in\mathcal{M}(n,0)$ and let $G$ be a closed subgroup of $\mathrm{Isom}(Y)$. Suppose that\\
	(1) $Gy$ is $\delta$-geodesic in $Y$,\\
	(2) $d_{GH}((Y,y,Gy),(\mathbb{R}^k\times X,(0,x),\mathbb{R}^k\times Z))\le\delta,$ where $X$ is a length space and $Z$ is a closed subset in $X$.\\
	(3) $\mathrm{diam}(Z)\ge L$.\\
	Then 
	$$d_{GH}((Y,y,Gy),(\mathbb{R}^{k+1}\times X',(0,x'),\mathbb{R}^{k+1}\times Z'))\le \Phi(\delta,L^{-1}|n)$$
	for some length metric $(X',x')$ and some closed geodesic subset $Z'\subseteq X'$. 
\end{lem}

\begin{proof}
	Suppose that there is a contradicting sequence $\{(Y_i,y_i,G_i)\}_i$ with\\
	(1) $G_iy_i$ is $\delta_i$-geodesic in $Y_i$, where $\delta_i\to 0$;\\
	(2) $d_{GH}((Y_i,y_i,G_i y_i),(\mathbb{R}^k\times X_i,(0,x_i),\mathbb{R}^k\times Z_i))=\delta_i\to 0$,\\
	(3) $\mathrm{diam}(Z_i)\to\infty$.\\ 
	Passing to a subsequence, 
	$$(Y_i,y_i,G_i)\overset{GH}\longrightarrow (Y_\infty,y_\infty,G_\infty)$$
	such that
	$$(Y_\infty,y_\infty,G_\infty y_\infty)=(\mathbb{R}^k\times X_\infty,(0,x_\infty),\mathbb{R}^k\times Z_\infty),$$
	where $X_\infty$ is a length space. By Lemma \ref{to_geodesic} and the hypothesis (1,3), $Z_\infty$ is a non-compact closed geodesic subset in $X_\infty$. We show that $\mathbb{R}^k\times X_\infty$ indeed splits off an $\mathbb{R}^{k+1}$-factor. Suppose that $X_\infty$ does not contain any line. For each $j$, let $z_j\in \{0\}\times Z_\infty$ such that $d(z_j,y_\infty)\ge j$ and let $\gamma_j$ be a minimal geodesic from $y_\infty$ to $z_j$ that lies in $\{0\}\times Z_\infty\subseteq G_\infty y_\infty$. We write the midpoint of $\gamma_j$ as $h_j\cdot y_\infty$ for some $h_j\in G_\infty$. Then $h_j^{-1}\gamma_j$ is a sequence of minimal geodesics which have midpoint $y_\infty$ and have length $\ge j\to\infty$. $h_j^{-1}\gamma_j$ sub-converges to a line in $Y_\infty=\mathbb{R}^k\times X_\infty$. Note that each $h_j$ maps to $\{0\}\times X_\infty$ to itself, thus this limit line is in the $X_\infty$-factor. It follows from the Cheeger-Colding splitting theorem that $\mathbb{R}^k\times X_\infty$ is isometric to $\mathbb{R}^{k+1}\times X'$. 
	
	It remains to show that the orbit $G_\infty y_\infty=\mathbb{R}^k\times Z_\infty$ is indeed $\mathbb{R}^{k+1}\times Z'$ for some $Z'\subseteq X'$. Since the minimal geodesics $h_j^{-1}\gamma_j$ are in $\{0\}\times Z_\infty$, its limit line is contained in $\{0\}\times Z_\infty$ as well. Combined with the fact that $\mathbb{R}^k\times Z_\infty$ is geodesic in $\mathbb{R}^k\times X_\infty=\mathbb{R}^{k+1}\times X'$, we see that $\mathbb{R}^{k}\times Z_\infty$ contains a slice $\mathbb{R}^{k+1}$. Applying Lemma \ref{geo_in_product}, we obtain the desired conclusion.
\end{proof}

By Proposition \ref{noncpt_orbit_split}, for any $s>0$ and any $(Y,y,G)\in\Omega(\widetilde{M},\Gamma)$, we have
$$d_{GH}((sY,y,Gy),(\mathbb{R}\times X_s, (0,x_s),\mathbb{R}\times Z_s))\le \Phi(\epsilon|n)$$
for some length space $X_s$ and a closed geodesic subset $Z_s\subseteq X_s$. If $Z_s$ has large or infinite diameter for all $s>0$, then we shall apply the above Lemma \ref{large_diam_split} to split off an $\mathbb{R}^2$-factor for all $(sY,y,Gy)$.

\begin{lem}\label{small_large_cpt}
	Let $(Y,y)\in\mathcal{M}(n,0)$ and let $G$ be a closed subgroup of $\mathrm{Isom}(Y)$. Let $\delta\in (0,10^{-4})$. Suppose that for all $s>0$, we have
	$$d_{GH}((sY,y,Gy),(\mathbb{R}^k\times X_s,(0,x_s),\mathbb{R}^k\times Z_s))\le\delta,$$
	where $(X_s,x_s)$ is a length space and $Z_s$ is a closed geodesic subset of $X_s$. 
	Then for $L=\delta^{-1/2}$, one of the followings holds:\\
	(1) $\mathrm{diam}(Z_s)>L$ for all $s>0$;\\
	(2) $\mathrm{diam}(Z_s)<L^{-1}$ for all $s>0$;\\
	(3) there is $s>0$ such that $\mathrm{diam}(Z_s)\in [0.9,1.1]$, given that $\delta$ is sufficiently small. 
\end{lem}

\begin{proof}
	Suppose that $(Y,y,G)\in\Omega(\widetilde{M},\Gamma)$ does not belong to cases (1) and (2) listed in the statement. This means that there exists $s>0$ such that $(sY,y,G)$ satisfies $d=\mathrm{diam}(Z_s)\in[L^{-1},L].$
	
	We claim that $(d^{-1}sY,y,G)$ has $\mathrm{diam}(Z_{d^{-1}s})\in [0.9,1.1]$. Let $$F:sY\to \mathbb{R}^k\times X_s$$ be an $\delta$-approximation between $(sY,y,Gy)$ and $(\mathbb{R}^k\times X_s,(0,x_s),\mathbb{R}^k\times Z_s)$. We consider its scaling $$d^{-1}F:d^{-1}sY\to \mathbb{R}^k\times (d^{-1}X_s).$$ When $d^{-1}\ge 1$, $d^{-1}F$ is a $d^{-1}\delta$-approximation between $(d^{-1}sY,y,Gy)$ and $(\mathbb{R}^k\times (d^{-1}X_s),(0,x_s),\mathbb{R}^k\times (d^{-1}Z_s))$; when $d^{-1}\le 1$, $d^{-1}F$ is a $d\delta$-approximation between the above two spaces. Since $d\in[L^{-1},L]= [\delta^{1/2},\delta^{-1/2}]$, we see that $d^{-1}F$ shows
	$$d_{GH}((d^{-1}sY,y,Gy),(\mathbb{R}^k\times (d^{-1}X_s),(0,x_s),\mathbb{R}^k\times (d^{-1}Z_s)))\le \delta^{1/2},$$
	where $d^{-1}Z_s$ has diameter $1$. On the other hand, we have assumption 
	$$d_{GH}((d^{-1}sY,y,Gy),(\mathbb{R}^k\times X_{d^{-1}s}, (0,x_{d^{-1}s}),\mathbb{R}^k\times Z_{d^{-1}s}))\le \delta.$$
	Therefore,
	$$\mathrm{diam}(Z_{d^{-1}s})=\mathrm{diam}(d^{-1}Z_s)\pm \Phi(\delta)=1\pm \Phi(\delta).$$
\end{proof}

\begin{proof}[Proof of Proposition \ref{split_all_scales}]
	To avoid confusions, we will write $\Phi_0$ as the estimate in Lemma \ref{large_diam_split} and $\Phi_1$ as the one in Proposition \ref{noncpt_orbit_split}, respectively.
	
	By Proposition \ref{noncpt_orbit_split}, $(Y,y,G)$ satisfies the conditions in Lemma \ref{small_large_cpt} with $k=1$ and $\delta=\Phi_1(\epsilon|n)$. If $(Y,y,G)$ belongs to case (3) of Lemma \ref{small_large_cpt}, then we are done. For case (2), we have
	$$d_{GH}((sY,y,Gy),(\mathbb{R}\times X_s,(0,x_s),\mathbb{R}\times \{x_s\}))\le\Phi'_1(\epsilon|n)$$
	with a slightly increased $\Phi'_1(\epsilon|n)$. Case (1) is where we shall continue the splitting process. It follows from Lemma \ref{large_diam_split} that for all $s>0$,
	$$d_{GH}((sY,y,Gy),(\mathbb{R}^2\times X'_s,(0,x'_s),\mathbb{R}^2\times Z'_s))\le\Phi_0(\Phi_1,\Phi_1^{1/2}|n)=:\Phi_2(\epsilon|n).$$
	Applying the above procedure repeatedly with Lemmas \ref{large_diam_split} and \ref{small_large_cpt}, we end in the desired result.
\end{proof}	
	
\section{Limit orbits as almost Euclidean spaces}\label{sec_eu}
	
In this section, we first rule out case (2) in Proposition \ref{split_all_scales}; this shows that the orbit $Gy$ are almost Euclidean for all $(Y,y,G)\in\Omega(\widetilde{M},\Gamma)$ (Theorem \ref{almost_eu_orbit}). Then we prove Theorem A.

The proof of Theorem \ref{almost_eu_orbit} uses a critical rescaling argument, which is effective in proving uniformity among all spaces in $\Omega(\widetilde{M},\Gamma)$ (see \cite{Pan_eu,Pan_al_stable,Pan_escape}).	

We introduce a notation for convenience. 

\begin{defn}\label{def_approx}
	Let $(Y,y,G)\in\Omega(\widetilde{M},\Gamma)$. In the context of Proposition \ref{split_all_scales}, we call $k$ as the \textit{approximate Euclidean dimension} of $Gy$, written as $\mathrm{EuDim}_A(Gy)$; we also call $\mathrm{diam}(Z_1)$ as an \textit{approximate width} of $Gy$, written as $\mathrm{Wid}_A(Gy)$. 
\end{defn}

For a fixed space $(Y,y,G)\in\Omega(\widetilde{M},\Gamma)$, note that $\mathrm{EuDim}_A(Gy)$ is uniquely determined, while $\mathrm{Wid}_A(Gy)$ may allow a small error up to $\Phi(\epsilon|n)$.
	
\begin{proof}[Proof of Theorem \ref{almost_eu_orbit}]
	We show that for any $(Y,y,G)\in\Omega(\widetilde{M},\Gamma)$, there is an integer $k$ such that
	$$d_{GH}((Y,y,Gy),(\mathbb{R}^k\times X,(0,x),\mathbb{R}^k\times \{x\}))\le \Phi(\epsilon|n).$$
	If we can prove this, then by Proposition \ref{cpt_cnt}, $k$ must be uniform among all spaces in $\Omega(\widetilde{M},\Gamma)$. 
	
	$(Y,y,G)$ has two possibilities as listed in the Proposition \ref{split_all_scales}. It suffices to rule out case (2), that is, there is $s>0$ such that
	$$d_{GH}((sY,y,Gy),(\mathbb{R}^k\times X_s,(0,x_s),\mathbb{R}^k\times Z_s))\le\Phi(\epsilon|n)$$
	with $\mathrm{diam}(Z_s)\in [0.9,1.1]$. We also assume that $(Y,y,G)$ has the smallest $k=\mathrm{EuDim}_A(Gy)$ so that case (2) occurs; in other words, if a space $(W,w,H)\in\Omega(\widetilde{M},\Gamma)$ also belongs to case (2), then $\mathrm{EuDim}_A(Hw)\ge k$. Besides $(sY,y,G)$, we shall also consider its scaling $(10sY,y,G)\in \Omega(\widetilde{M},\Gamma)$. By Proposition \ref{split_all_scales} and the proof of Lemma \ref{small_large_cpt}, we have
	$$d_{GH}((10sY,y,Gy),(\mathbb{R}^k\times X_{10s},(0,x_{10s}),\mathbb{R}^k\times Z_{10s}))\le\Phi(\epsilon|n)$$
	with $\mathrm{diam}(Z_{10s})=10\mathrm{diam}(Z_s)\pm\Phi(\epsilon|n)\in [8,12]$. 
	
	For convenience, we now write 
	$$(Y_1,y_1,G_1):=(10sY,y,G),\quad (Y_2,y_2,G_2):=(sY,y,G);$$
	correspondingly; we also write 
	\begin{align*}
	(\mathbb{R}^k\times V_1,(0,v_1),\mathbb{R}^k\times U_1)&:=(\mathbb{R}^k\times X_{10s},(0,x_{10s}),\mathbb{R}^k\times Z_{10s}),\\
	(\mathbb{R}^k\times V_2,(0,v_2),\mathbb{R}^k\times U_2)&:=(\mathbb{R}^k\times X_{s},(0,x_{s}),\mathbb{R}^k\times Z_{s}).
	\end{align*}
	Note that we have
	$$\mathrm{diam}(U_1)\in [8,12],\quad \mathrm{diam}(U_2)\in [0.9,1.1].$$
	Since both $(Y_1,y_1,G_1)$ and $(Y_2,y_2,G_2)$ are equivariant asymptotic cones of $(\widetilde{M},\Gamma)$, there are sequences $r_i,s_i\to\infty$ such that
	$$(r_i^{-1}\widetilde{M},\tilde{p},\Gamma)\overset{GH}\longrightarrow (Y_1,y_1,G_1),\quad (s_i^{-1}\widetilde{M},\tilde{p},\Gamma)\overset{GH}\longrightarrow (Y_2,y_2,G_2).$$
	After passing to a subsequence, we assume that
	$$t_i:=s_i^{-1}/r_i^{-1}\to\infty.$$
	Setting $(N_i,q_i,\Gamma_i)$ as $(r_i^{-1}\widetilde{M},\tilde{p},\Gamma)$, we have
	$$(N_i,q_i,\Gamma_i)\overset{GH}\longrightarrow (Y_1,y_1,G_1),\quad (t_iN_i,q_i,\Gamma_i)\overset{GH}\longrightarrow (Y_2,y_2,G_2).$$
	
	Next we choose an intermediate scaling sequence $l_i$ as follows. For each $i$, let
	\begin{align*}
	L_i=\{1\le l\le t_i\ |\ &d_{GH}((lN_i,q_i,\Gamma_i),(W,w,H))\le 10^{-3} \text{ for some space}\\
	& (W,w,H)\in\Omega(\widetilde{M},\Gamma) \text{ such that $\mathrm{EuDim}_A(Hw)<k$, or}\\
	& \mathrm{EuDim}_A(Hw)=k \text{ with } \mathrm{Wid}_A(Hw)\le 2 \}.
	\end{align*}
	Since $(Y_2,y_2,G_2)$ satisfies $\mathrm{EuDim}_A(G_2y_2)=k$ and $\mathrm{Wid}_A(G_2y_2)\le 1.1$, it is clear that $t_i\in L_i$ for all $i$ large. We choose $l_i\in L_i$ such that $\inf L_i\le l_i \le \inf L_i+1/i$. 
	
	\textbf{Claim 1:} $\liminf l_i>\Phi(\epsilon|n)^{-1/2}$, where $\Phi(\epsilon|n)$ is the constant in Proposition \ref{split_all_scales}. We argue by contradiction. Suppose that $l_i\to l\in [1,\Phi(\epsilon|n)^{-1/2}]$ for a subsequence. Then
	$$(l_iN_i,q_i,\Gamma_i)\overset{GH}\longrightarrow (lY_1,y_1,G_1).$$
	For each $i$, since $l_i\in L_i$, there is $(W_i,w_i,H_i)\in\Omega(\widetilde{M},\Gamma)$ such that
	$$d_{GH}((l_iN_i,q_i,\Gamma_i),(W_i,w_i,H_i))\le 10^{-3};$$
	moreover, 
	$$d_{GH}((W_i,w_i,H_i),(\mathbb{R}^{m_i}\times X_i,(0,x_i),\mathbb{R}^{m_i}\times Z_i))\le \Phi(\epsilon|n)$$
	with $m_i\le k$, and $\mathrm{diam}(Z_i)\le 2$ if $m_i=k$.
	Recall that
	$$d_{GH}((Y_1,y_1,G_1\cdot y_1),(\mathbb{R}^k\times V_1,(0,v_1),\mathbb{R}^k\times U_1))\le\Phi(\epsilon|n);$$
	let $F$ be an $\Phi(\epsilon|n)$ approximation between them. Then its scaling $lF$ shows that
	$$d_{GH}((lY_1,y_1,G_1\cdot y_1),(\mathbb{R}^k\times lV_1,(0,v_1),\mathbb{R}^k\times lU_1))\le l\cdot \Phi(\epsilon|n)\le \Phi(\epsilon|n)^{1/2}.$$
	Combining the above together, we see that
	$$d_{GH}((\mathbb{R}^{m_i}\times X_i,(0,x_i),\mathbb{R}^{m_i}\times Z_i),(\mathbb{R}^k\times lV_1,(0,v_1),\mathbb{R}^k\times lU_1))\le 10^{-2}+\Phi'(\epsilon|n)$$
	for $i$ large. If $m_i<k$, then by our choice of $k$, $\mathrm{diam}(Z_i)=0$ and thus the above estimate cannot hold. If $m_i=k$, then $\mathrm{diam}(Z_i)\le 2$ and $\mathrm{diam}(lU_1)\ge 8$ also lead to a contradiction. This proves Claim 1.
	
	Next we consider the convergence under the critical rescaling $l_i$:
	$$(l_iN_i,q_i,\Gamma_i)\overset{GH}\longrightarrow (Y',y',G')\in \Omega(\widetilde{M},\Gamma).$$
	
	\textbf{Claim 2:} $\mathrm{EuDim}_A(G'y')\le k$; $\mathrm{Wid}_A(G'y')\le 3$ when $\mathrm{EuDim}_A(G'y')=k$. By Proposition \ref{split_all_scales}, for each $s>0$, we have
	$$d_{GH}((sY',y',G'y'),(\mathbb{R}^{m'}\times X'_{s},(0,x'_s),\mathbb{R}^{m'}\times Z'_{s}))\le\Phi(\epsilon|n),$$
	where $m'=\mathrm{EuDim}_A(G'y').$ Similar to the proof of Claim 1, for each $i$ we can find $(W_i,w_i,H_i)$ that is $10^{-3}$-close to $(l_iN_i,q_i,\Gamma_i)$, and $(\mathbb{R}^{m_i}\times X_i,(0,x_i),\mathbb{R}^{m_i}\times Z_i)$ that is $\Phi(\epsilon|n)$-close to $(W_i,w_i,H_i)$. This shows that
	$$d_{GH}((\mathbb{R}^{m_i}\times X_i,(0,x_i),\mathbb{R}^{m_i}\times Z_i),(\mathbb{R}^{m'}\times X'_{1},(0,x'_1),\mathbb{R}^{m'}\times Z'_{1}))\le 10^{-2}+\Phi'(\epsilon|n)$$
	for all $i$ large. Recall that either $m_i<k$ with $\mathrm{diam}(Z_i)=0$, or $m_i=k$ with $\mathrm{diam}(Z_i)\le 2$. We conclude that $m'<k$ or $m'=k$ with $\mathrm{diam}(Z'_1)\le 3$. This proves Claim 2.
	
	To reach a final contradiction, we consider 
	$$(10^{-1}l_iN_i,q_i,\Gamma_i)\overset{GH}\longrightarrow (10^{-1}Y',y',G').$$
	Note that the scaled down limit $(10^{-1}Y',y',G')$ satisfies $\mathrm{EuDim}_A(10^{-1}G'y')<k$ with $\mathrm{Wid}_A(10^{-1}G'y')=0$, or $\mathrm{EuDim}_A(10^{-1}G'y')=k$ with $\mathrm{Wid}_A(10^{-1}G'y')\le 1$. In other words, $(10^{-1}Y',y',G')$ satisfies the condition described in the definition of $L_i$. Since $\liminf l_i>\Phi(\epsilon|n)^{-1/2}$ as showed in Claim 1, we have $10^{-1}l_i\in [1,t_i]$ for all $i$ large. It follows that $10^{-1}l_i\in L_i$ for all $i$ large, which contradicts our choice of $l_i\in L_i$ as $\inf L_i\le l_i\le \inf L_i+1/i$. This completes the proof.
\end{proof}	
	
We show that the converse of Theorem \ref{almost_eu_orbit} also holds.

\begin{prop}\label{converse}
	Let $(M,p)$ be an open manifold of $\mathrm{Ric}\ge 0$. Suppose that there is an integer $k$ such that for any $(Y,y,G)\in\Omega(\widetilde{M},\Gamma)$, we have
	$$d_{GH}((Y,y,Gy),(\mathbb{R}^k\times X,(0,x),\mathbb{R}^k\times \{x\}))\le \eta,$$
	where $(X,x)$ is a length space that depends on $(Y,y)$. Then $E(M,p)\le \Phi(\eta|n)$.
\end{prop}

We prove a lemma below before proving Proposition \ref{converse}.

\begin{lem}\label{geo_almost_prod}
	Let $(Y,y)\in\mathcal{M}(n,0)$ and let $N$ be a closed subset of $Y$ containing $y$. Suppose that
	$$d_{GH}((Y,y,N),(\mathbb{R}^k\times X,(0,x),\mathbb{R}^k\times \{x\}))\le \eta$$ 
	for some length space $(X,x)$. Then for any point $p\in N$ with $d(p,y)\le 10$ and any minimal geodesic $\sigma$ from $y$ to $p$, $\gamma$ must be contained in the $\Phi(\eta|n)$-neighborhood of $N$. 
\end{lem}

\begin{proof}
	We argue by contradiction. Suppose that there are $\delta>0$ and a sequence of spaces $(Y_i,y_i,N_i)$ with
	$$d_{GH}((Y_i,y_i,N_i),(\mathbb{R}^{k_i}\times X_i,(0,x_i),\mathbb{R}^{k_i}\times \{x_i\}))\to 0$$
	but for some minimal geodesic $\sigma_i$ from $y_i$ to some point $p_i\in N_i$ with $d(y_i,p_i)\le 10$, $\sigma_i$ is not contained in the $\delta$-neighborhood of $N_i$. After passing to a subsequence, we have convergence
	$$(Y_i,y_i,N_i)\overset{GH}\longrightarrow (Y_\infty,y_\infty,N_\infty)=(\mathbb{R}^k\times X_\infty,(0,x_\infty),\mathbb{R}^k\times \{x_\infty\}).$$
	For this subsequence, we can also assume that $p_i\to p_\infty \in \mathbb{R}^k\times \{x_\infty\}$ and $\sigma_i$ converges to a minimal geodesic $\sigma_\infty$ from $(0,x_\infty)$ to $p_\infty$. By hypothesis, $\sigma_\infty$ is not contained in the $\delta/2$-neighborhood of $N_\infty$. On the other hand, as a segment in the metric product, $\sigma_\infty$ must be contained in $\mathbb{R}^k\times \{x_\infty\}=N_\infty$. A contradiction.
\end{proof}
	
\begin{proof}[Proof of Proposition \ref{converse}]
	We write $\epsilon=E(M,p)$. Let $\gamma_i\in \pi_1(M,p)$ be a sequence with $r_i=d(\gamma_i\tilde{p},\tilde{p})\to\infty$ and $c_i$ be a sequence of representing geodesic loops based at $p$ such that
	$$\epsilon_i:=\dfrac{d_H(x,c_i)}{r_i}\to \epsilon.$$
	Let $\sigma_i$ be the lift of $c_i$ in $\widetilde{M}$ starting from $\tilde{p}$. Then by our choice $\sigma_i$ is not contained in $\pi^{-1}(B_{\epsilon_ir_i/2}(p))$, where $\pi:(\widetilde{M},\tilde{p})\to (M,p)$ is the covering map. For a convergent subsequence
	\begin{center}
		$\begin{CD}
		(r^{-1}_i\widetilde{M},\tilde{p},\Gamma,\gamma_i) @>GH>> 
		(Y,y,G,g)\\
		@VV\pi V @VV\pi V\\
		(r^{-1}_iM,p) @>GH>> (Z=Y/G,z),
		\end{CD}$
	\end{center}
    it is clear that $d(y,gy)=1$. We can also assume that $\sigma_i$ converges to a minimal geodesic $\sigma$ from $y$ to $gy$. We also know that $\sigma$ is not contained in $\pi^{-1}(B_{\epsilon/3}(z))$. On the other hand, by assumption and Lemma \ref{geo_almost_prod}, $\sigma$ should be contained in a $\Phi(\eta|n)$-neighborhood of $Gy$, that is, $\pi^{-1}(B_{\Phi(\eta|n)}(z))$. This shows that $\epsilon\le 3\Phi(\eta|n)$.
\end{proof}

Combined with the results previously, we can also obtain the converse of Proposition \ref{limit_delta_geodesic}.

\begin{cor}\label{cor_converse}
	Let $(M,p)$ be an open manifold of $\mathrm{Ric}\ge 0$ with an infinite fundamental group. If the orbit $Gy$ is $\eta$-geodesic for all $(Y,y,G)\in\Omega(\widetilde{M},\Gamma)$, then $E(M,p)\le \Phi(\eta|n)$.
\end{cor}

\begin{proof}
	From the proof of Proposition \ref{split_all_scales} and Theorem \ref{almost_eu_orbit}, we see that $Gy$ being $\eta$-geodesic from all $(Y,y,G)\in\Omega(\widetilde{M},\Gamma)$ implies that
	$$d_{GH}((Y,y,Gy),(\mathbb{R}^k\times X,(0,x),\mathbb{R}^k\times \{x\}))\le \Phi(\eta|n),$$
	where $(X,x)$ is a length space depending on $(Y,y)$. Together with Proposition \ref{converse}, the result follows.
\end{proof}

%\begin{rem}
%	One can formulate a slightly different version of Corollary \ref{cor_converse}. Instead of assuming that $Gy$ is $\eta$-geodesic for all $(Y,y,G)$, we can assume the following: for any $(Y,y,G)\in\Omega(\widetilde{M},\Gamma)$, there is a length space $(X,x)$ and a closed geodesic subset $N$ of $X$ such that
%	$$d_{GH}((Y,y,Gy),(X,x,N))\le\eta.$$ 
%\end{rem}

\begin{cor}
	Let $(M,p)$ be an open manifold of $\mathrm{Ric}\ge 0$ with an infinite fundamental group. If $E(M,p)\le \epsilon$, then $E(M,q)\le \Phi(\epsilon|n)$ for all $q\in M$.
\end{cor}

\begin{proof}
	We write $\Gamma_p=\pi_1(M,p)$ and $\Gamma_q=\pi_1(M,q)$ for convenience. Note that
	$$\Gamma_p\cdot \tilde{p}=\pi^{-1}(p)=\Gamma_q\cdot \tilde{p},$$
	where $\pi:\widetilde{M}\to M$ is the covering map. Let $r_i\to\infty$ be any sequence. With respect to the convergence
	$$(r_i^{-1}\widetilde{M},\tilde{p},\Gamma_p)\overset{GH}\longrightarrow (Y,y,G),$$
	$\Gamma_p\cdot \tilde{p}=\Gamma_q\cdot \tilde{p}$ converges to $Gy$. By Theorem \ref{almost_eu_orbit}, there is an integer $k$ and a length metric space $(X,x)$ such that 
	$$d_{GH}((Y,y,Gy),(\mathbb{R}^k\times X,(0,x),\mathbb{R}^k\times \{x\}))\le \Phi(\epsilon|n).$$
	Under the scaling $r_i^{-1}$, $\Gamma_q\cdot \tilde{q}$ and $\Gamma_q\cdot \tilde{p}$ converges to the same limit. As a result, the conditions in Proposition \ref{converse} hold for $(M,q)$ with $\eta=\Phi(\epsilon|n)$. Applying Proposition \ref{converse}, we conclude that $E(M,q)\le \Phi'(\epsilon|n)$.
\end{proof}

We move on to prove Theorem A. We first show that any transitive nilpotent group action on an almost Euclidean orbit is by almost translations, which is a quantitative version of \cite[Lemma 3.10]{Pan_escape}.
	
\begin{lem}\label{almost_trans}
	Let $(Y,y)\in \mathcal{M}(n,0)$ and $G$ be a closed subgroup of $\mathrm{Isom}(Y)$. Suppose that\\
	(1) $d_{GH}((Y,y,Gy),(\mathbb{R}^k\times X,(0,x),\mathbb{R}^k\times\{x\}))\le \delta$ for some length space $(X,x)$,\\
	(2) $G$ has a closed nilpotent subgroup $H$ of nilpotency length $\le n$ that acts transitively on $Gy$.\\
	Then 
	$$d(h^2y,y)\ge (2-\Phi(\delta|n))\cdot d(hy,y)$$
	for all $h\in H$ with $d(hy,y)\le 10$.
\end{lem}	

\begin{proof}
	Suppose that $(Y_i,y_i,G_i)$ is a sequence of spaces such that for each $i$,\\
	(1) $d_{GH}((Y_i,y_i,G_i y_i),(\mathbb{R}^{k_i}\times X_i,(0,x_i),\mathbb{R}^{k_i}\times\{x_i\}))\le \delta_i\to 0$ for some length space $X_i$,\\
	(2) $G_i$ has a closed nilpotent subgroup $H_i$ of nilpotency length $\le n$ acting transitively on $G_i\cdot y_i$.\\
	Since $k_i\le n$, without lose of generality, we can assume that all $k_i$ are the same, denoted as $k$. Then passing to a subsequence, we obtain convergence
	$$(Y_i,y_i,G_i,H_i)\overset{GH}\longrightarrow (Y_\infty,y_\infty,G_\infty,H_\infty).$$
	It follows from (1) that $$(Y_\infty,y_\infty,G_\infty y_\infty)=(\mathbb{R}^k\times X_\infty,(0,\infty),\mathbb{R}^k\times\{x_\infty\}).$$
	Passing (2) to the limit, we see that $H_\infty$ is a closed nilpotent group acting transitively on $G_\infty \cdot y_\infty=\mathbb{R}^k\times\{x_\infty\}$. By \cite[Lemma 3.10]{Pan_escape}, $H_\infty$ acts as translations on $G_\infty\cdot y_\infty$. In particular, 
	$$d(h_\infty^2 y_\infty,y_\infty)=2\cdot d(h_\infty y_\infty,y_\infty)$$
	holds for all $h_\infty\in H_\infty$. This immediately implies the desired result.
\end{proof}

Theorem \ref{almost_eu_orbit} and Lemma \ref{almost_trans} restrict $\Gamma$-action on $\widetilde{M}$ at large scale if $E(M,x)\le\epsilon$.

\begin{lem}\label{almost_trans_large}
	Given $n$, there is $\epsilon(n)>0$ such that the following holds.
	
	Let $(M,p)$ be an open manifold of $\mathrm{Ric}\ge 0$ and $E(M,p)\le\epsilon(n)$. Let $N$ be a nilpotent subgroup of $\Gamma$ of finite index and nilpotency length $\le n$. Then there is $R>0$, depending on $M$, such that
	$$|\gamma^2|\ge 1.9\cdot |\gamma|$$
	for all $\gamma\in N$ with $|\gamma|\ge R$.
\end{lem}

\begin{proof}
    We argue by contradiction. Suppose that there is a sequence $\gamma_i\in N$ with $r_i=d(\gamma_i\tilde{p},\tilde{p})\to \infty$ and 
    $$d(\gamma_i^2\tilde{p},\tilde{p})< 1.9\cdot d(\gamma_i\tilde{p},\tilde{p}).$$
    We consider the convergence
    $$(r_i^{-1}\widetilde{M},\tilde{p},\Gamma,N,\gamma_i)\overset{GH}\longrightarrow(Y,y,G,H,h).$$
    $H$ is a closed nilpotent subgroup of $G$ with finite index and nilpotency length $\le n$. By Corollary \ref{cor_cnt_orbit}, the orbit $Gy$ is connected. Hence $H$ acts transitively on $Gy$. Then it follows from Theorem \ref{almost_eu_orbit} and Lemma \ref{almost_trans} that 
    $$d(h^2y,y)\ge (1-\Phi(\epsilon|n))\cdot d(hy,y).$$
    When $\epsilon$ is small that $\Phi(\epsilon|n)\le 0.01$, we see that
    $$d(\gamma_i^2\tilde{p},\tilde{p})\ge 1.95 \cdot d(\gamma_i\tilde{p},\tilde{p})$$
    for all $i$ large, which is a contradiction to our assumption.
\end{proof}

\begin{proof}[Proof of Theorem A]
	By \cite[Lemma 2.1]{Pan_escape}, $\Gamma$ is finitely generated. Then $\Gamma$ has a nilpotent subgroup of $\Gamma$ of finite index with nilpotency length at most $n$ \cite{Mil,Gro_poly,KW}. It follows from the same argument in \cite[Theorem A]{Pan_escape} (also see \cite[Section 4]{Pan_al_stable}) that the conclusion in Lemma \ref{almost_trans_large} implies that $N$ has an abelian subgroup of finite index.
\end{proof}

\end{document}